\documentclass[a4]{amsart}

\usepackage{amssymb}
\usepackage{amsfonts}
\usepackage{mathtools}
\usepackage[utf8]{inputenc}
\usepackage[english]{babel}
\usepackage{enumitem}
\setenumerate[1]{label=(\alph*)}
\setenumerate[2]{label=(\roman*)}
\usepackage[all]{xy}
\usepackage{tikz}
\usetikzlibrary{matrix,arrows}
\usepackage{hyperref}

\newtheorem{theorem}{Theorem}
\newtheorem{proposition}[theorem]{Proposition}
\newtheorem{lemma}[theorem]{Lemma}
\theoremstyle{definition}
\newtheorem{definition}[theorem]{Definition}

\newtheorem{remark}[theorem]{Remark}

\numberwithin{theorem}{section}
\numberwithin{equation}{section}

\newcommand{\define}[1]{\emph{#1}}

\newcommand{\kk}{{k}}

\newcommand{\smpr}{{\operatorname{sp}}}
\newcommand{\sqp}{{\operatorname{sq}}}
\newcommand{\cs}{{\operatorname{cs}}}
\newcommand{\smpp}{{\operatorname{sp}}}
\DeclareMathOperator{\HH}{H}
\DeclareMathOperator{\BB}{B}
\DeclareMathOperator{\GL}{GL}

\DeclareMathOperator{\K}{K}
\DeclareMathOperator{\DM}{DM}

\DeclareMathOperator{\SB}{SB}
\DeclareMathOperator{\Var}{Var}
\DeclareMathOperator{\Spec}{Spec}
\DeclareMathOperator{\id}{id}

\newcommand{\sptag}[1]{\href{http://stacks.math.columbia.edu/tag/#1}{Tag~#1}}

\makeatletter
\newcommand*{\da@rightarrow}{\mathchar"0\hexnumber@\symAMSa 4B }
\newcommand*{\da@leftarrow}{\mathchar"0\hexnumber@\symAMSa 4C }
\newcommand*{\xdashrightarrow}[2][]{%
  \mathrel{%
    \mathpalette{\da@xarrow{#1}{#2}{}\da@rightarrow{\,}{}}{}%
  }%
}
\newcommand{\xdashleftarrow}[2][]{%
  \mathrel{%
    \mathpalette{\da@xarrow{#1}{#2}\da@leftarrow{}{}{\,}}{}%
  }%
}
\newcommand*{\da@xarrow}[7]{%
  \sbox0{$\ifx#7\scriptstyle\scriptscriptstyle\else\scriptstyle\fi#5#1#6\m@th$}%
  \sbox2{$\ifx#7\scriptstyle\scriptscriptstyle\else\scriptstyle\fi#5#2#6\m@th$}%
  \sbox4{$#7\dabar@\m@th$}%
  \dimen@=\wd0 %
  \ifdim\wd2 >\dimen@
    \dimen@=\wd2 %
  \fi
  \count@=2 %
  \def\da@bars{\dabar@\dabar@}%
  \@whiledim\count@\wd4<\dimen@\do{%
    \advance\count@\@ne
    \expandafter\def\expandafter\da@bars\expandafter{%
      \da@bars
      \dabar@ 
    }%
  }%
  \mathrel{#3}%
  \mathrel{%
    \mathop{\da@bars}\limits
    \ifx\\#1\\%
    \else
      _{\copy0}%
    \fi
    \ifx\\#2\\%
    \else
      ^{\copy2}%
    \fi
  }%
  \mathrel{#4}%
}
\makeatother

\hypersetup{  
  pdftitle    = {A Bittner presentation for Deligne--Mumford stacks},
  pdfauthor   = {Daniel Bergh},
  colorlinks=true}

\title[Grothendieck group of DM-stacks]{Weak factorization and the Grothendieck group of Deligne--Mumford stacks}
\begin{document}
\author{Daniel Bergh}
\address{%
Department of Mathematical Sciences\\
University of Copenhagen\\
Universitetsparken~5\\
DK-2100 Copenhagen Ø,
Denmark.}
\email{dbergh@math.ku.dk}
\begin{abstract}
We construct a presentation for the Grothendieck group of Deligne--Mumford stacks over a field of characteristic zero.
The generators for this presentation are smooth, proper Deligne--Mumford stacks and the relations are expressed in terms of stacky blow-ups.
In the process we prove a version of the weak factorization theorem for Deligne--Mumford stacks.
\end{abstract}

\maketitle

\section{Introduction}
Throughout the article, we fix a field $\kk$ of characteristic 0.
All algebraic stacks are understood to be of finite presentation over $\kk$.

In \cite{bittner2004}, Heinloth (née~Bittner) constructs a presentation for the Grothendieck group
$\K_0(\Var_\kk)$ of varieties with generators which are smooth, proper varieties,
and relations in terms of blow-ups.
In this paper, we generalize Heinloth's result, and give a \emph{Bittner presentation} for the
Grothendieck group $\K_0(\DM_\kk)$ of Deligne--Mumford stacks (Theorem~\ref{t-main}).

A key ingredient in Heinloth's proof is the \emph{weak factorization theorem},
originally by~Włodarczyk~\cite{wlodarczyk2003},
which states that any birational map between smooth, proper varieties factors
as a sequence of blow-ups and blow-downs in smooth centers.

In birational geometry for Deligne--Mumford stacks, 
\emph{stacky blow-ups} (see Definition~\ref{def-stacky-blow-up})
play a similar role as usual blow-ups do in birational geometry for varieties.
A generalization of the weak factorization theorem states that any birational
map between smooth, proper Deligne--Mumford stacks factors
as a sequence of stacky blow-ups and stacky blow-downs in smooth centers.
A proof of this theorem has been given by Rydh (unpublished),
and, very recently, independently by Harper \cite{harper2017}.
In this article we also formulate and prove a version of the theorem (Theorem~\ref{th-intro-wf-dm}).
The proof relies on \emph{equivariant weak factorization} by Abramovich et~al. \cite{abramovich2002}
and \emph{functorial destackification} by the author~\cite{bergh2017}.
See Remark~\ref{rem-attribution} for a more detailed account on how
the results by Rydh and Harper relate to the results presented in this article.

A ring homomorphism with $\K_0(\Var_\kk)$ as domain is sometimes called a \emph{motivic measure}.
The Bittner presentation gives a convenient way construct motivic measures from invariants defined on
smooth and proper varieties,
and several interesting results on the structure of $\K_0(\Var_\kk)$ depends on this technique.
As an example,
Larsen and Lunts \cite{ll2003} showed that
the association which takes a smooth proper variety to its stably birational class extends
to a motivic measure $\K_0(\Var_\kk) \to \mathbb{Z}[\SB_\kk]$
whose kernel is generated by the class of the affine line.
Here $\mathbb{Z}[\SB_\kk]$ denotes the monoid ring on the monoid $\SB_\kk$ of stably birational classes of $k$-varieties.
This invariant was used by Poonen~\cite{poonen2002} to construct zero divisors in $\K_0(\Var_\mathbb{Q})$,
and later by Borisov~\cite{borisov2018} to show that in fact the class of the affine line is a zero divisor in $\K_0(\Var_\kk)$.
As another example of motivic measures constructed using the Bittner presentation,
Ekedahl~\cite{ekedahl2009gg} considered the invariant $\HH^\bullet(-, \mathbb{Z}_\ell)$ taking a smooth proper variety to its
$\ell$-adic cohomology groups.
He used this motivic measure to show that for any non-special, connected algebraic group $G$,
there exists a $G$-torsor $P \to X$ such that $\{G\}\{X\} \neq \{P\}$ in $\K_0(\Var_\kk)$~\cite{ekedahl2009ap}.

It is the hope of the author that the results presented here will have similar
impact on the understanding of $\K_0(\DM_k)$ as Heinloth's result has on the understaning of~$\K_0(\Var_k)$.
An application is given by the author jointly with Gorchinskiy--Larsen--Lunts~\cite{bgll2017}.
Here we extend the \emph{categorical measure},
originally defined on~$\K_0(\Var_k)$ by Bondal--Larsen--Lunts~\cite{bll2004},
to a motivic measure defined on $\K_0(\DM_k)$.
The categorical measure takes vales in the Grothendieck group of saturated $k$-linear differential graded categories
and the measure, in principle, associates a smooth, proper Deligne--Mumford stack to the class of its derived category of
perfect complexes.
We use this measure to prove a conjecture of Galkin--Shinder regarding the so called
\emph{categorical zeta function}~\cite[Proposition~6.8]{bgll2017}.
An independent proof of this is given by Gyenge in \cite{gyenge2017}, also using the results from this article.
In \cite[Section~6.4]{bgll2017}, the categorical measure on $\K_0(\DM_k)$ is also used to shed some light on a conjecture by Polishchuk--Van den Bergh regarding a semi-orthogonal decomposition
of a certain equivariant derived category~\cite[Conjecture~A]{pvdb2015}.

\subsection{Main results}
The Grothendieck group $\K_0(\DM_\kk)$ of Deligne--Mumford stacks is defined as the group freely generated by
equivalence classes $\{X\}$ of Deligne--Mumford stacks $X$ subject to the \emph{scissors relations}
\begin{equation}
\label{eq-scissor}
\{X\} = \{X \setminus Z\} + \{Z\}
\end{equation}
where $Z \subset X$ is a closed substack.
The group $\K_0(\DM_\kk)$ also has a ring structure with multiplication determined by
\begin{equation}
\label{eq-multiplication}
\{X\}\cdot \{Y\} := \{X \times Y\}
\end{equation}
for Deligne--Mumford stacks $X$ and $Y$.

Consider the group $\K_0(\DM^\smpp_\kk)$ generated by equivalence classes $\{X\}$ of smooth and proper Deligne--Mumford stacks $X$ subject to the \emph{stacky blow-up relations}
\begin{equation}
\label{eq-blow-up-rel}
\{\widetilde{X}\} - \{E\} = \{X\} - \{Z\}, \qquad \{\emptyset\} = 0
\end{equation}
where $\widetilde{X}$ is a stacky blow-up of a smooth, proper Deligne--Mumford stack
$X$ in a smooth center $Z$, and $E$ is the exceptional divisor.
This group has a multiplicative structure,
defined by \eqref{eq-multiplication} for smooth and proper Deligne--Mumford stacks $X$ and $Y$,
and we have an obvious ring homomorphism
\begin{equation}
\label{eq-bittner-presentation}
\K_0(\DM^\smpp_\kk) \to \K_0(\DM_\kk), \qquad \{X\} \mapsto \{X\},
\end{equation}
where $X$ is a smooth proper Deligne--Mumford stack.
Our main result on $\K_0(\DM_\kk)$ is the following theorem.
\begin{theorem}[Bittner presentation for Deligne--Mumford stacks]
\label{t-main}
The ring homomorphism \eqref{eq-bittner-presentation} is an isomorphism.
\end{theorem}

The proof of Theorem~\ref{t-main} relies on the existence of weak factorizations for birational maps
between smooth, proper Deligne--Mumford stacks into stacky blow-ups and blow-downs.

\begin{theorem}[Weak factorization for Deligne--Mumford stacks]
\label{th-intro-wf-dm}
Let $\varphi\colon Y \dashrightarrow X$
be a proper birational map of smooth, separated Deligne--Mumford stacks of finite type
over a field $\kk$ of characteristic zero.
Assume that both $X$ and $Y$ are global quotient stacks (see Remark~\ref{rem-global-quotient}),
and let $U \subset X$ be an open substack over which $\varphi$ is an isomorphism.
Then $\varphi$ factors as a sequence
\begin{equation}
\label{eq-intro-wf}
Y := X_n \xdashrightarrow{\varphi_n} \cdots \xdashrightarrow{\varphi_1} X_0 =: X
\end{equation}
of proper birational maps $\varphi_i$ satisfying the following properties:
\begin{enumerate}
\item
\label{it-wf-stacky-blow-up}
either $\varphi_i$ or $\varphi_i^{-1}$ is a stacky blow-up centered in a smooth locus;
\item
\label{it-wf-invariant-locus}
each $\varphi_i$ is an isomorphism over $U$;
\item
\label{it-wf-normal-crossings}
if the exceptional loci $D_i := X_i\setminus U$ are simple normal crossing divisors for $i = 0, n$,
then the same holds true for all $i$,
and all stacky blow-ups have centers which have normal crossings with the respective~$D_i$.
\item
\label{it-wf-defined}
there is an integer $n_0$ such that $X_i \dashrightarrow X$ is everywhere defined whenever $i \leq n_0$ and
$X_i \dashrightarrow Y$ is everywhere defined whenever $i \geq n_0$;
\item
\label{it-wf-projective}
the induced birational morphisms $(X_i)_\cs \to X_\cs$ on coarse spaces is projective for $i \leq n_0$
and similarly for $(X_i)_\cs \to Y_\cs$ when $i \geq n_0$.
\end{enumerate}
\end{theorem}

\begin{remark}
\label{rem-attribution}
To the authors knowledge, a version of Theorem~\ref{th-intro-wf-dm} was first proven by Rydh,
who announced the result at a conference in Kiola, March~2016.
The proof by Rydh is an adaptation of the techniques from the original proof by Włodarczyk.
However, at the time of writing, Rydh's proof is not yet available in written form.

Being aware of the work by Rydh,
I found a different proof of Theorem~\ref{th-intro-wf-dm},
which I present in this paper.
It relies on equivariant weak factorization for varieties by Abramovich~et\ al.~\cite{abramovich2002},
together with functorial destackification by myself \cite{bergh2017}.
The proof is quite short owing to the fact that most of the work is done by the theorems
from the cited sources.

While still working on this article,
a proof of Theorem~\ref{th-intro-wf-dm} by Harper~\cite{harper2017}
appeared as a preprint.
Harper uses ideas similar to those presented here.
Since my proof is not long,
and in some aspects slightly different than Harper's,
I decided to nevertheless include my proof in this paper.
\end{remark}

\begin{remark}
\label{rem-global-quotient}
It should be stressed that both the result by Rydh and by Harper are
more general than Theorem~\ref{th-intro-wf-dm},
since they do not assume that the stacks $X$ and $Y$ be global quotient stacks
(i.e., of the form $[V/\GL_n]$ where $V$ is an algebraic space with an action by $\GL_n$).
However, many interesting classes of Deligne--Mumford stacks are known to be global quotients.
For instance, this holds for any orbifold due to Edidin--Hassett--Kresch--Vistoli~\cite[Theorem~2.18]{ehkv2001}
and for any smooth, separated Deligne--Mumford stack with quasi-projective coarse space due to Kresch~\cite[Theorem ~4.4]{kresch2009}.
To my knowledge,
there are no known examples of smooth Deligne--Mumford stacks which are not global quotients.
Hence it is unclear how restrictive this assumption is.
\end{remark}

\subsection{Preliminaries}
The following standard results on algebraic stacks and birational geometry will be used freely.

Due to a classical result by Keel--Mori \cite{km1997},
any algebraic stack $X$ with finite inertia admits a \emph{coarse space}, which we denote by $X_\cs$.
More generally, since the formation of the coarse space respects flat base change,
it is easy to see that any morphism $X \to S$ of algebraic stacks with finite relative inertia
\cite[\sptag{036X}]{stacks-project} factors through the \emph{relative coarse space} $X \to \underline{X} \to S$.
This factorization is characterized by the property that it is initial among factorizations
$X \to Y \to S$ with $Y \to S$ representable by an algebraic space.

Since we are working over a field of characteristic zero,
any reduced algebraic stack admits a \emph{resolution of singularities}.
This follows from the functorial version of Hironaka's theorem on resolution
of singularities for varieties \cite{hironaka1964}, by Bierstone--Milman \cite{bm1997}.
Moreover, we will freely use \emph{functorial embedded resolution of singularities}
and \emph{functorial principalization}.
See \cite[Section~1.2]{abramovich2002} for a nice overview of these concepts.

\subsection{Outline}
In Section~\ref{basic}, we give some elementary results on alternative presentations for $\K_0(\DM_k)$.
In Section~\ref{weak}, we prove the weak factorization theorem Deligne--Mumford stacks
(Theorem~\ref{th-intro-wf-dm}).
Finally, we give the proof for the Bittner presentation for $\K_0(\DM_k)$ (Theorem~\ref{t-main})
in Section~\ref{bittner}.

\subsection{Acknowledgments}
I would like to thank Sergey~Gorchinskiy, Valery Lunts and David~Rydh for valuable comments.
 
\section{Alternative presentations}
\label{basic}
In this section we give some alternative presentations for $\K_0(\DM_k)$ which can be obtained without using weak factorization. 

Let $\mathcal{P}$ be a class of Deligne--Mumford stacks.
We let $\K_0(\DM_k^\mathcal{P})$ denote the group generated by isomorphism classes $\{X\}$
of objects $X$ in $\mathcal{P}$ subject to the scissors relations \eqref{eq-scissor} with respect to
closed immersions $Z \subset X$ of objects in $\mathcal{P}$.
It is clear that we have group homomorphism
\begin{equation}
\label{eq-comparison}
\K_0(\DM_k^\mathcal{P}) \to \K_0(\DM_k)
\end{equation}
induced by the inclusion of $\mathcal{P}$ into the class of all Deligne--Mumford stacks.
Moreover, if $\mathcal{P}$ is closed under products in the category of Deligne--Mumford stacks,
it is clear that the group $\K_0(\DM_k^\mathcal{P})$ has an induced ring structure,
making the homomorphism \eqref{eq-comparison} a ring homomorphism.

\begin{proposition}
\label{pr-alt-def}
Let $\mathcal{P}$ be a class of Deligne--Mumford stacks.
Assume that every Deligne--Mumford stack admits a stratification into objects in $\mathcal{P}$.
Then \eqref{eq-comparison} is an isomorphism.

In particular, this holds if we take $\mathcal{P}$ to be the class of Deligne--Mumford stacks $X$ such that either
\begin{enumerate}
\item
\label{it-gerbe}
$X$ is a gerbe;
\item
\label{it-affine}
$X$ has coarse space which is affine;
\item
\label{it-smooth}
$X$ is smooth;
\item
\label{it-global-finite}
$X$ is a global quotient by a finite, constant group;
\item
\label{it-finite-inertia}
$X$ has finite inertia;
\item
\label{it-separated}
$X$ is separated;
\item
\label{it-global}
$X$ is a global quotient stack. That is, we have $X \cong [Y/\GL_n]$ where $Y$ is an algebraic space endowed with a $\GL_n$-action;
\item
\label{it-quasi-projective}
$X$ is quasi-projective in the sense of \cite[Definition~5.5]{kresch2009}.
That is, $X \cong [Y/\GL_n]$ where $Y$ is a quasi-projective variety and $\GL_n$ acts linearly and properly on $X$\cite[Proposition~5.1, Theorem~5.3]{kresch2009}.
\end{enumerate}
\end{proposition}
\begin{proof}
It is an easy exercise to verify that the map
$$
\K_0(\DM_k) \to \K_0(\DM^\mathcal{P}_k), \qquad \{X\} \mapsto \sum_i \{X_i\},
$$
where $X$ is a Deligne--Mumford stack and $X_i$ is
a stratification of $X$ into objects in $\mathcal{P}$,
is a well defined inverse to the homomorphism~\eqref{eq-comparison}.

Next we prove the other part of the statement.
Let $X$ be a Deligne--Mumford stack.
Recall that we assume that $X$ is of finite presentation over a field $k$.
In particular, the stack $X$ is quasi-compact and quasi-separated.
It is harmless to assume that $X$ is reduced.
By noetherian induction,
it suffices to show that $X$ generically satisfies all the listed statement.
By \cite[\sptag{06RC}]{stacks-project},
we may assume that $X$ is a gerbe,
so \ref{it-gerbe} is satisfied.
By definition of gerbe, $X$ has a coarse space~$X_\cs$.
Under our assumptions, $X_\cs$ is quasi-separated algebraic space.
Hence $X_\cs$ has a non-empty schematic locus by \cite[\sptag{06NH}]{stacks-project}.
After base change of $X$ to an open affine in $X_\cs$,
we may assume that \ref{it-affine} holds.
Since $X$ is reduced, the same holds for $X_\cs$.
Hence $X_\cs$ is generically smooth.
Since the canonical morphism form a gerbe to its coarse space is smooth,
we may therefore assume that \ref{it-smooth} holds.
Now let $\pi\colon U \to X_\cs$ be an étale trivializing cover for the gerbe $X$.
By shrinking $U$ and $X_\cs$ if necessary,
we may assume that $\pi$ is finite by \cite[\sptag{06NH}]{stacks-project}.
By construction, $\pi$ lifts to a finite étale covering $U \to X$,
which we may assume to have constant degree $n$.
By taking the $n$-fold fiber product of $U$ over $X$ and removing the
big diagonal, we obtain a torsor for the symmetric group $S_n$ which
covers $X$.
This shows that \ref{it-global-finite} holds.

Now it is straightforward to verify that the rest of the properties hold.
Indeed, \ref{it-global-finite} implies \ref{it-finite-inertia} and \ref{it-global}.
Furthermore \ref{it-affine} and \ref{it-finite-inertia} imply \ref{it-separated}.
Finally \ref{it-affine} and \ref{it-global} imply~\ref{it-quasi-projective}
by \cite[Proposition~5.1]{kresch2009}. 
\end{proof}

The Grothendieck group of Deligne--Mumford stacks also admits a rather down to earth 
presentation in terms of equivariant varieties.
Let $R$ be the group generated by isomorphism classes $\{X\}^G$,
where $G$ is a finite, constant group and $X$ is variety on which $G$ acts properly.
As usual, we impose the \emph{scissors relations}
\begin{equation}
\label{eq-eq-scissor}
\{X\}^G = \{X \setminus Z\}^G + \{Z\}^G
\end{equation}
where $Z \subset X$ is a $G$-invariant closed subvariety.
We also impose the \emph{change of group relations}
\begin{equation}
\label{eq-eq-group-change}
\{X\}^{G\times H} = \{X/G\}^H,
\end{equation}
where $G$ and $H$ are finite, constant groups and $X$ is a variety endowed with a $G\times H$
action such that the restriction of the action to $G$ is free.
\begin{remark}
Note that the quotient $X/G$ in the relation \eqref{eq-eq-group-change} need not be a variety in the usual sense,
but rather an algebraic space.
When regarding Grothendieck groups,
this is harmless since any $H$-equivariant algebraic spaces admits an $H$-equivariant
stratification into $H$-equivariant varieties.
Indeed, this can be proven similarly as Proposition~\ref{pr-alt-def}.
\end{remark}

We endow the group $R$ with a ring structure by letting
\begin{equation}
\label{eq-eq-mult}
\{X\}^G\cdot \{Y\}^H := \{X\times Y\}^{G\times H}
\end{equation}
and extending by linearity.
It is easy to verify that this expression respects the relations \eqref{eq-eq-scissor} and \eqref{eq-eq-group-change}.

\begin{proposition}
\label{pr-eq-pres}
We have an isomorphism
$
R \xrightarrow{\sim} \K_0(\DM_k)
$
of rings, which is determined by
$
\{X\}^G \mapsto \{[X/G]\}
$
where $G$ is a finite, constant group acting on a variety $X$.
\end{proposition}
\begin{proof}
First we verify that the association $\{X\}^G \mapsto \{[X/G]\}$ gives a well defined ring homomorphism.
The relation \eqref{eq-eq-scissor} is satisfied by the usual scissors relations for $\K_0(\DM_k)$
since $[Z/G] \subset [X/G]$ is a closed substack with complement $[(X\setminus Z)/G]$.
The relations \eqref{eq-eq-group-change} and \eqref{eq-eq-mult} are satisfied since we
have isomorphisms $[X/(G\times H)] \cong [(X/G)/H]$ and $[(X \times Y)/(G\times H)] \cong [X/G]\times [Y/H]$
of algebraic stacks.
Thus we get an induced ring homomorphism $R \to \K_0(\DM_k)$
which we denote by~$\varphi$.

Let us construct an inverse $\varphi^{-1}$ to this map.
By Proposition~\ref{pr-alt-def}~\ref{it-global-finite},
it is enough to define $\varphi^{-1}$ on stacks which are global quotients by finite, constant groups.
The association $\varphi^{-1}\colon \{[X/G]\} \mapsto \{X\}^G$ clearly respects the scissors
relations.
If $[X/G] \cong [Y/H]$,
we may form the 2-cartesian square
$$
\xymatrix{
Z \ar[r]\ar[d] & X \ar[d] \\
Y \ar[r] & [X/G] \\
}
$$
where the morphisms from $X$ and $Y$ are the coverings induced by the $G$- and $H$-actions, respectively.
It follows that $Z$ is a $G\times H$-torsor over $[X/G]$, a $G$-torsor over $Y$ and an $H$-torsor over $X$.
Hence, by applying the change of groups relation~\eqref{eq-eq-group-change} twice, we get
$$
\{X\}^G = \{Z\}^{G \times H} = \{Y\}^H,
$$
so $\varphi^{-1}$ is well defined.
\end{proof}

\section{The weak factorization theorem}
\label{weak}
In this section, we develop weak factorization for Deligne--Mumford stacks that are global quotients.
\begin{definition}
\label{def-stacky-blow-up}
Let $X$ be an algebraic stack, let $Z \subset X$ be a closed substack,
and let $r$ be a positive integer.
A \define{stacky blow-up of $X$ with center $Z$ and weight $r$}
is a morphism $\pi\colon \widetilde{X} \to X$ composed by a usual blow-up of $X$
in $Z$ followed by an $r$-th root construction in the exceptional divisor.
The \emph{exceptional divisor} $E$ of a stacky blow-up is the universal $r$-th
root of the divisor $\pi^{-1}(Z)$.
Given a simple normal crossings divisor $D \subset X$,
the \define{total transform} of $D$ to $\widetilde{X}$ is $\widetilde{D} + E$
where $\widetilde{D}$ is the birational transform of $D$ to $\widetilde{X}$.
Note that by contrast to the situation with ordinary blow-ups,
the total transform is not the pull-back of $D$ to $\widetilde{X}$.
\end{definition}

The key property of stacky blow-ups which we need here is that a stacky blow-up of
a smooth algebraic stack $X$ in a smooth center $Z \subset X$ is smooth,
and that the total transform of any simple normal crossings divisor $D \subset X$
having normal crossings with $Z$ is a simple normal crossings divisor in~$\widetilde{X}$.

We make the following \emph{ad hoc} definitions in order to encapsulate our standard assumptions
on Deligne--Mumford stacks endowed with simple normal crossings divisors in a convenient way.
\begin{definition}
\label{def-stacky-standard}
Let $S$ be an algebraic stack.
Recall that we assume that all algebraic stacks are of finite presentation over a field $\kk$ of characteristic zero.
\begin{enumerate}
\item
A \define{standard pair} $(X, D)$ over $S$ is a smooth, separated morphism $X \to S$
of algebraic stacks, which is Deligne--Mumford in the relative sense (see~\cite[\sptag{04YW}]{stacks-project}),
together with an effective Cartier divisor $D \subset X$
which has simple normal crossings relative~$S$.
That is, the divisor $D$ is a sum of effective Cartier divisors which are smooth over $S$,
and $D$ is a simple normal crossings divisor over each fiber of $X \to S$.
\item
\label{it-birational-map}
A \define{birational map} $\varphi \colon (Y, E) \dashrightarrow (X, D)$ of standard pairs over $S$
is a birational map $Y \dashrightarrow X$ inducing an isomorphism $Y\setminus E \to X \setminus D$ over $S$.
Such a map is \define{proper} if the normalization of $X \times_S Y$ in the diagonal
$U \cong X\setminus D \cong Y \setminus E \to X\times_S Y$ is proper over both
$X$ and $Y$.
If $\varphi$ is everywhere defined over $S$,
then we say that $\varphi$ is a \define{birational morphism over $S$}.
Sometimes, we indicate this with a solid arrow.
\item
\label{it-stacky-blow-up}
A birational map as in \ref{it-birational-map} is an \define{admissible stacky blow-up} if there is
a locus $Z \subset D \subset X$,
having normal crossings with $D$ relative $S$,
and a weight $r \geq 1$
such that $Y$ is the stacky blow up of $X$ in $Z$ with weight $r$,
and $E$ is the total transform of $D$.
In this situation, we also say that $\varphi^{-1}$ is an \emph{admissible stacky blow-down}.
\item
An admissible stacky blow-up $\varphi$ as in \ref{it-stacky-blow-up} is simply called
an \emph{admissible blow-up} if the weight $r = 1$.
\item
An admissible stacky blow-up $\varphi$ as in \ref{it-stacky-blow-up} is simple called
an \emph{admissible root construction} if $Z$ is a smooth divisor.
\end{enumerate}
\end{definition}

\begin{remark}
For our purposes, we only need to consider the cases when $S$ is either~$\Spec \kk$
or a classifying stack $\BB G$ for some linear algebraic group $G$.
In fact, it is enough to consider $G = \GL_n$.
For a smooth algebraic stack $X$ over $\BB G$,
it is easy to see that a divisor $D \subset X$ is a simple normal crossings divisor
over $\BB G$ if and only if it is a simple normal crossings divisor over $\Spec \kk$.
\end{remark}

First we recall the equivariant weak factorization theorem by
Abramovich et~al. \cite{abramovich2002}.
Here is a formulation, in the language of algebraic stacks, of a version of the theorem that suffices for our needs.
\begin{theorem}[Equivariant weak factorization]
\label{th-equiv-wf}
Let $G$ be an algebraic group,
and let $\varphi\colon (Y, E) \to (X, D)$ be a proper birational morphism of standard pairs over the
classifying stack~$\BB G$.
Assume that both $X$ and $Y$ are representable over $\BB G$.
Then $\varphi$ factors as a sequence
$$
(Y, E) := (X_n, D_n) \xdashrightarrow{\varphi_n} \cdots \xdashrightarrow{\varphi_1} (X_0, D_0) =: (X, D)
$$
of proper birational maps $\varphi_i$ over $\BB G$,
where for each $i$ either $\varphi_i$ or $\varphi_i^{-1}$ is an admissible blow-up.
Each of the maps $X_i \dashrightarrow X$ is everywhere defined.
Moreover, the morphism $X_i \dashrightarrow X$ is projective if the same holds for $\varphi$.
\end{theorem}
\begin{proof}
The representable morphisms to $\BB G$ from $X$ and $Y$ tells us that $X \cong [X'/G]$ and $Y \cong [Y'/G]$
where $X'$ and $Y'$ are smooth algebraic spaces with $G$-actions.
Since the morphism $\varphi$ is defined over $\BB G$,
it lifts to a $G$-equivariant proper morphism $\varphi'\colon X' \to Y'$.
Hence the theorem follows from Theorem~0.3.1 together with the remark about $G$-equivariance
(Remark~2 below the formulation of the theorem) from \cite{abramovich2002}.
There are a few subtle points when verifying that the theorem applies. 
\begin{enumerate}
\item
Remark 2 from \cite[Section~0.3]{abramovich2002} also applies
to algebraic actions by algebraic groups,
and not only for actions by discrete groups.
We give a short sketch to motivate this claim.
For simplicity, we assume that $k$ is algebraically closed,
and that $X'$ and $Y'$ are schemes,
leaving the general case to the reader.
By Theorem~0.3.1 we have a $G(k)$-equivariant weak factorization
$$
X'_n \xdashrightarrow{\varphi'_n} \cdots \xdashrightarrow{\varphi'_1} X'_0.
$$
We need to verify that each action morphism $G(k)\times X'_i \to X'_i$ is induced by
a morphism $\alpha_i \colon G \times X'_i \to X'_i$.
Assume, by induction, that we have constructed $\alpha_i$.
We need to treat two cases.

\emph{Case 1:} $\varphi'_{i+1}$ is a blow-up.
The center of the blow-up is $G$-invariant,
as this can be verified on $k$-points.
Hence the action morphisms $\alpha_i$ lifts canonically to an action morphism
$\alpha_{i+1}$.

\emph{Case 2:} $(\varphi'_{i+1})^{-1}$ is a blow-up.
Consider the diagram
$$
\xymatrix{
G \times X'_i \ar[r]^{\alpha_i}\ar[d]_{\gamma} & X'_i \ar[d]^{(\varphi'_{i+1})^{-1}}\\
G \times X'_{i + 1} \ar@{-->}[r]_{\alpha_{i+1}} & X'_{i + 1},
}
$$
where $\gamma = \id\times(\varphi'_{i+1})^{-1}$.
Since $(\varphi'_{i+1})^{-1}$ is $G(k)$-equivariant,
it follows that the composition $(\varphi'_{i+1})^{-1}\circ \alpha_i$ is constant on the fibers
of $\gamma$.
Note that $\gamma$ is a blow-up of a smooth variety in a smooth center.
In particular, the morphism $\gamma$ is closed and
$\varphi_\ast(\mathcal{O}_{G \times X'_i}) \cong \mathcal{O}_{G\times X'_{i+1}}$.
Hence \cite[Lemma~8.11.1]{egaII} applies,
so $(\varphi'_{i+1})^{-1}\circ \alpha_i$ factors through $\gamma$,
and it follows that $\alpha_{i+1}$ is a morphism.
\item
In the statement of \cite[Theorem~0.3.1]{abramovich2002},
the algebraic spaces $X'$ and $Y'$ are assumed to be proper,
which need not be the case in our situation.
However, the only way the properness is used in the proof is to reduce to the case when
the birational map is given by a proper birational morphism.
\item
In our setting, a \emph{$G$-invariant simple normal crossings divisor}
is a normal crossings divisor which is a sum of $G$-invariant smooth divisors.
It is easy to see that if the exceptional locus of $X'$ and $Y'$ has this property,
then the same is true for all intermediate steps.
Indeed, this follows from the fact that each center is smooth,
$G$-invariant and has simple normal crossings with the exceptional locus.
\end{enumerate}
\end{proof}

\begin{theorem}[Relative destackification]
\label{th-rel-destack}
Let $(X, D)$ be a standard pair over an algebraic stack $S$
such that the restriction of the structure map $X \to S$ to $X\setminus D$ is representable.
Then there exists a sequence
\begin{equation}
\label{eq-destack-sequence}
(X_n, D_n) \xrightarrow{\varphi_n} \cdots \xrightarrow{\varphi_1} (X_0, D_0) =: (X, D)
\end{equation}
of admissible stacky blow-ups over $S$ such that the relative coarse space $\underline{X}_n$ of $X_n$
over $S$ is smooth over~$S$.

Moreover, the relative coarse space $\underline{D}_n$ of $D_n$ over $S$ is a simple normal crossings divisor,
and the canonical birational morphism $(X_n, D_n) \to (\underline{X}_n, \underline{D}_n)$
factors as a sequence of admissible root constructions. 
\end{theorem}
\begin{proof}
If $S$ is representable, then this is a special case of~\cite[Theorem~1.2]{bergh2017}.
Since the destackification algorithm is functorial with respect to arbitrary change of base~$S$,
we obtain the relative version by a descent argument,
which we work out in detail below.
Before doing this, we would like to stress two subtle points about the functoriality.
\begin{enumerate}
\item
\label{it-components}
The destackification algorithm works with simple normal crossings divisors together
with a decomposition of the divisor into a sequence of smooth
(but not necessarily irreducible)
components (see~\cite[Definition~2.1]{bergh2017}).
Moreover, the algorithm is functorial with respect to this data.
\item
\label{it-roots}
The formulation of the definition of root construction in \cite[Section~2.2]{bergh2017}
allows roots to be taken along several components of the
simple normal crossings divisor simultaneously.
But since the components have a given ordering,
there is a canonical way to translate this into a sequence of root constructions in the
sense of Definition~\ref{def-stacky-standard}. 
\end{enumerate}

Let $S' \to S$ be a smooth surjective morphism with $S'$ a scheme.
Also let $S'' \rightrightarrows S'$ denote the groupoid formed by the
fibre product $S'' := S'\times_S S'$ together with the relevant maps.
Now consider the pull-backs $X''_0 \rightrightarrows X'_0$ and
$D''_0 \rightrightarrows D'_0$ of this groupoid to $X_0$ and $D_0$ respectively.
Then $(X'_0, D'_0)$ is a standard pair with $X'_0 \setminus D'_0$ representable.
Since the stacky locus of $X'_0$ is contained in a simple normal crossings divisor,
the stabilizers are abelian.
Indeed, this is true for global quotient stacks by \cite[Theorem~4.1]{ry2000},
and in general one can reduce to this case by working étale locally on the coarse space
(a more detailed account on this will be given in a forthcoming article jointly with Rydh).
The same holds for the standard pair~$(X''_0, D''_0)$.

Next we apply destackification to both $(X'_0, D'_0)$ and $(X''_0, D''_0)$.
Since the destackification algorithm is functorial with respect to arbitrary base change,
the sequence obtained by the destackification of $(X''_0, D''_0)$ is
isomorphic to the pull-back along any of the two projections $X''_0 \rightrightarrows X'_0$
of the sequence obtained by destackification of $(X'_0, D'_0)$.

Now we construct the sequence~\eqref{eq-destack-sequence} recursively.
At each step~$i$, with $n > i \geq 0$, we have a solid diagram
$$
\xymatrix{
X''_{i + 1} \ar[r]_{\varphi''_{i+1}} \ar@<-.8ex>[d] \ar@<.8ex>[d]
& X''_i \ar@<-.8ex>[d] \ar@<.8ex>[d]
& Z''_i \ar[l] \ar@<-.8ex>[d] \ar@<.8ex>[d]\\
X'_{i+1} \ar[r]_{\varphi'_{i+1}}\ar@{.>}[d] & X'_i\ar[d] & Z'_i \ar[l]\ar@{.>}[d]\\
X_{i+1} \ar@{.>}[r]_{\varphi_{i+1}} & X_i & Z_i \ar@{.>}[l]\\
}
$$
where the squares are cartesian in the sense described above,
and $Z'_i$ and $Z''_i$ are the centers of the respective stacky blow-ups at step $i$.
The center $Z'_i$ descends to a closed substack of $Z_i \subset X_i$ producing
the lower right-hand cartesian square.
In particular $Z_i$ is smooth and has simple normal crossings with $D_i$,
since both these properties can be checked locally.
Taking the stacky blow-up with the given weight produces the
lower left-hand cartesian square.
We also get $D_{i+1}$ as the total transform of $D_i$.

Finally we let $X_n \to \underline{X}_n \to S$ be the canonical factorization through the
relative coarse space.
Base change to $S'$ gives us the diagram
$$
\xymatrix{
X'_n \ar[d]\ar[r] & \underline{X}'_n \ar[d]\ar[r] & S' \ar[d] \\
X_n \ar[r] & \ar[r] \underline{X}_n & S
}
$$
with cartesian squares.
Each component, in the sense of \ref{it-components} in the beginning of the proof, of $D_n'$
is invariant under the groupoid $\underline{X}''_n \rightrightarrows \underline{X}'_n$.
Hence $\underline{D}_n' \subset X_n'$ descends to a simple normal crossings
divisor~$\underline{D}_n \subset X_n$.
That $X_n \to \underline{X}_n$ is an iterated root construction in the components of $\underline{D}_n$
can be checked locally on $\underline{X}_n$,
which concludes the proof.
\end{proof}

\begin{remark}
From the proof of Theorem~\ref{th-rel-destack},
we see that it is unnecessary to put any conditions on the algebraic stack $S$
besides that $S$ should be quasi-compact
(and this is only needed to ensure that the destackification sequence is finite).
It is also clear that we have the same functoriality properties as in the original theorem.
In other words, Theorem~\ref{th-rel-destack} is a direct generalization of \cite[Theorem~1.2]{bergh2017}.
\end{remark}

\begin{theorem}[Weak factorization for Deligne--Mumford stacks]
\label{th-wf-dm}
Let $\varphi\colon (Y, E) \dashrightarrow (X, D)$
be a proper birational map of standard pairs over $\kk$.
Assume that both $X$ and $Y$ are global quotient stacks.
Then $\varphi$ factors as a sequence
\begin{equation}
\label{eq-wf}
(Y, E) := (X_n, D_n) \xdashrightarrow{\varphi_n} \cdots \xdashrightarrow{\varphi_1} (X_0, D_0) =: (X, D)
\end{equation}
of proper birational maps $\varphi_i$,
where for each $i$ either $\varphi_i$ or $\varphi^{-1}_i$ is an admissible stacky blow-up.
\end{theorem}
\begin{proof}
Since $X$ and $Y$ are global quotient stacks, we may choose representable morphisms
$X \to \BB G$ and $Y \to \BB H$ where $G$ and $H$ are linear algebraic groups.
Let $U \cong X \setminus D \cong Y \setminus E$ and let $\widetilde{U}$ be the normalization
of $X \times Y$ in the diagonal morphism $U \to X \times Y$.
Since $U \to U \times U$ is finite,
by the assumption that $U$ is separated,
it follows that the obvious morphisms from $\widetilde{U}$ to $X$ and $Y$ are isomorphisms over $U$.
By using functorial resolution of singularities on $\widetilde{U}$ together with functorial principalization
on the complement of $U$,
we get a standard pair $(Z, F)$ such that $Z\setminus F \cong U$.
By construction, we have proper birational morphisms
$\alpha\colon (Z, F) \to (X, D)$ and $\beta\colon (Z, F) \to (Y, E)$ defined over
$\BB G$ and $\BB H$, respectively.

Now we apply relative destackification,
as described in Theorem~\ref{th-rel-destack},
to $(Z, F)$ twice ---
once relative $\BB G$ and once relative~$\BB H$.
This puts us in the situation described in the following diagram
\begin{equation}
\label{eq-weak-diagram}
\vcenter{
\xymatrix{
(X', D') \ar[d]_{\gamma'}\ar[dr]^{\alpha'} & & (Y', E') \ar[dl]_{\beta'}\ar[d]^{\delta'} \\
(\underline{X}', \underline{D}') \ar[d]_\gamma & (Z, F) \ar[dl]_\alpha\ar[dr]^\beta & (\underline{Y}', \underline{E}') \ar[d]^{\delta}\\
(X, D) & & (Y, E) \\
}
}
\end{equation}
Here the morphism $\alpha'$ is the composition of the admissible stacky blow-ups in the 
destackification sequence \eqref{eq-destack-sequence},
and $\gamma'$ is the canonical morphism to the relative coarse space over $\BB G$.
Recall that $\gamma'$ factors as a sequence of admissible root constructions.
Since $X$ is representable over $\BB G$,
we obtain a canonical morphism $\gamma$ making the left-hand side of the diagram commute.
Now $\gamma$ is a representable proper morphism over $\BB G$,
so Theorem~\ref{th-equiv-wf} applies and $\gamma$ factors weakly as a sequence
of admissible blow-ups.
Since the right-hand side of diagram \eqref{eq-weak-diagram} can be treated similarily,
but now relative $\BB H$,
we obtain the desired weak factorization of $\varphi$.
\end{proof}

The version of the weak factorization theorem stated above
suffices to prove Theorem~\ref{t-main}.
To get a full proof of Theorem~\ref{th-intro-wf-dm}, we need to work a little bit more.

\begin{proof}[Proof of Theorem~\ref{th-intro-wf-dm}]
Given a proper birational map $\varphi \colon Y \dashrightarrow X$,
as in the statement of the theorem,
we reduce to the case where we have a proper birational map of standard
pairs as in the statement of Theorem~\ref{th-wf-dm} by applying principalization
to the exceptional locus of both $Y$ and $X$.
It is now clear from Theorem~\ref{th-wf-dm} that we get a weak factorization
which satisfies \ref{it-wf-stacky-blow-up}, \ref{it-wf-invariant-locus} and~\ref{it-wf-normal-crossings}
in the statement of Theorem~\ref{th-intro-wf-dm}.
Furthermore, it is easy to see that it satisfies \ref{it-wf-defined} by inspecting
diagram~\eqref{eq-weak-diagram},
giving $(Z, F)$ index $n_0$,
and using that each stack in the equivariant weak factorization of $\gamma$ (resp.~$\delta$) is defined
over~$X$ (resp.~$Y$) as stated in Theorem~\ref{th-equiv-wf}.

Finally, to also show \ref{it-wf-projective}, we need to choose $(Z, F)$ more carefully.
Consider the diagram
\begin{equation}
\vcenter{
\xymatrix{
X'_\cs \ar[d]_{\gamma'_\cs}\ar[dr]^{\alpha'_\cs} & & Y'_\cs \ar[dl]_{\beta'_\cs}\ar[d]^{\delta'_\cs} \\
\underline{X}'_\cs \ar[d]_{\gamma_\cs} & Z_\cs \ar[dl]_{\alpha_\cs}\ar[dr]^{\beta_\cs} & \underline{Y}'_\cs \ar[d]^{\delta_\cs}\\
X_\cs & & Y_\cs \\
}
}
\end{equation}
obtained from \eqref{eq-weak-diagram} by passing to coarse spaces.
First we note that a representable projective morphism of separated Deligne--Mumford
stacks induce a projective morphism on the coarse spaces (see \cite[Proposition~2]{mo206117}).
Since a root construction induces an isomorphism on the coarse spaces,
it follows that $\alpha'_\cs$ and~$\beta'_\cs$ are projective and that $\gamma'_\cs$ and~$\delta'_\cs$
are isomorphisms.
Since \ref{it-wf-projective} holds for equivariant weak factorization
it thus suffices to show that $(Z, F)$ can be chosen such that $\alpha_\cs$ and~$\beta_\cs$
are projective.
To achieve this, we first construct $\widetilde{U}$ as in the proof of Theorem~\ref{th-wf-dm}.
By applying the version of Chow's Lemma due to Gruson--Raynaud~\cite[Corollaire~5.7.14]{gr1971} twice,
we obtain a birational modification $W \to \widetilde{U}_\cs$, which is an isomorphism over $U_\cs$,
such that both the induced morphisms $W \to X_\cs$ and~$W \to Y_\cs$ are projective.
Now, let $Z'$ be a desingularization of the pull back of $Z$ along $W \to Z_\cs$.
Applying principalization to the exceptional locus gives a standard pair $(Z', F')$
with the desired properties.
Now replacing $(Z, F)$ by $(Z', F')$ concludes the proof.
\end{proof}

\section{The Bittner presentation}
\label{bittner}
In this section, we prove Theorem~\ref{t-main}.
That is, we prove that the ring homomorphism
\begin{equation}
\label{eq-bittner}
\K_0(\DM^{\smpr}_k) \xrightarrow{\varphi} \K_0(\DM_k), \qquad \{X\} \mapsto \{X\},
\end{equation}
where $X$ is a smooth, proper Deligne--Mumford stack,
is an isomorphism.
We will follow the proof of the corresponding theorem for varieties given in~\cite{bittner2004}
closely.

Let $X$ be a smooth, proper Deligne--Mumford stack and let $D$ be a simple normal crossings divisor on $X$.
We will define a class $\{X, D\}$ in $\K_0(\DM^{\smpr}_k)$ associated to such a pair.
For the purpose of this definition, it will be convenient to generalize the notion of simple normal crossings
divisor to include the situation where $D$ is allowed to have smooth components of codimension~$0$.
Let $\{D_i\}$ be a finite family of smooth (generalized) divisors on $X$ indexed by a set $I$,
such that $D = \bigcup_{i \in I} D_i$.
Note that we neither require the divisors $D_i$ to be irreducible nor distinct.
Given a subset $J \subset I$,
we let $D_J$ denote the intersection $\bigcap_{i \in J} D_i$.
Here we use the convention $D_\emptyset = X$.
Now we define the expression
\begin{equation}
\label{eq-inc-exl}
\{X, D\} := \sum_{J \subset I} (-1)^{|J|} D_J
\end{equation}
in $\K_0(\DM^{\smpr}_k)$.
It is clear that the image of $\{X, D\}$ in $\K_0(\DM_k)$ under the homomorphism
$\varphi$ defined in \eqref{eq-bittner} is equal to $\{X\setminus D\}$.
Indeed, this follows from a simple argument using the principle of inclusion and exclusion.
The notation is motivated by the following lemma.

\begin{lemma}
The element $\{X, D\}$ in $\K_0(\DM^{\smpr}_k)$,
as defined above,
does not depend on the particular choice of family $\{D_i\}_{i \in I}$.
\end{lemma}
\begin{proof}
Since $\K_0(\DM_k)$ is additive for disjoint unions,
it suffices to prove that the expression \eqref{eq-inc-exl} is well defined as a
formal sum of smooth, irreducible, closed substacks of $X$.
But clearly the group of such formal sums injects into the group $\mathbb{Z}^{|X|}$
via the map taking a substack to the characteristic function of its underlying topological subspace.
In particular, the expression only depends on $D = \bigcup_{i \in I} D_i$
by the principle of inclusion and exclusion.
\end{proof}

Next, we investigate how expressions of the form $\{X, D\}$ are transformed by stacky
blow-ups.

\begin{proposition}
\label{pr-transform}
Let $(X, D)$ be a standard pair and let $Z \subset X$ be a smooth substack
having normal crossings with $D$.
Let $\widetilde{X}$ be a stacky blow-up of $X$ in $Z$ and let $\widetilde{D}$ be the total
transform of $D$ to $\widetilde{X}$.
Then
$$
\{\widetilde{X}, \widetilde{D}\} = \{X, D\} - \{Z, Z \cap D\}
$$ in $\K_0(\DM^{\smpr}_k)$.
In particular, if the stacky blow-up is admissible in the sense that $Z \subset D$,
then $\{\widetilde{X}, \widetilde{D}\} = \{X, D\}$.
\end{proposition}
\begin{proof}
We index the components $\widetilde{D}_i$ of $\widetilde{D}$ by the disjoint
union $\widetilde{I} = \{\bullet\} \cup I$,
such that $\widetilde{D}_i$ is the strict transform of $D_i$ if $i \in I$
and $\widetilde{D}_\bullet$ is the exceptional divisor of the stacky blow-up.
Rearranging the terms in the expression for $\{\widetilde{X},
\widetilde{D}\}$ yields
\begin{equation}
\label{eq-stacky-blow-up}
\{\widetilde{X}, \widetilde{D}\}
=
\sum_{J \subset I} (-1)^{|J|}
\left(\widetilde{D}_{J} - \widetilde{D}_{J\cup\{\bullet\}}\right).
\end{equation}
For $J \subset I$, we let $Z_J$ denote the intersection $Z \cap D_J = \bigcap_{i \in J}(Z \cap D_i)$.
Note that $\widetilde{D}_J$ is the stacky blow up of $D_J$ in the center
$Z_J$, and that $\widetilde{D}_{J \cup\{\bullet\}}$ is the exceptional
divisor of this blow-up.
Using the defining relations for $\K_0(\DM^{\smpr}_k)$,
it follows that the expression \eqref{eq-stacky-blow-up} equals
$$
\sum_{J \subset I} (-1)^{|J|}
\left(D_{J} - Z_{J}\right)
=
\{X, D\}
- \{Z, Z \cap D\},
$$
as desired.
\end{proof}

We are now ready to put the pieces together and prove that the ring homomorphism $\varphi$
defined in \eqref{eq-bittner} is an isomorphism.
\begin{proof}[Proof of Theorem~\ref{t-main}]
Let $\K_0(\DM^\sqp_k)$ denote the Grothendieck group generated by smooth,
quasi-projective Deligne--Mumford stacks subject to the usual scissors relations~\eqref{eq-scissor}.
We define a map
\begin{equation}
\label{eq-section}
\K_0(\DM^\sqp_k) \xrightarrow{\psi} \K_0(\DM^\smpr_k), \qquad \{U\} \mapsto \{X, X \setminus U\},
\end{equation}
where $U$ is a smooth quasi-projective Deligne--Mumford stack and $X$ is a compactification of $U$
by a projective Deligne--Mumford stack such that the exceptional locus $D := X \setminus U$
is a simple normal crossings divisor.
Such a compactification exists by~\cite[Theorem~5.3]{kresch2009}
combined with functorial resolution of singularities and principalization.

We need to show that $\psi$ is well defined.
Let $(X', D')$ be another compactification of $U$.
Since $X$ and $X'$ are global quotient stacks,
Theorem~\ref{th-wf-dm} applies and the
birational map $(X, D) \dashrightarrow (X', D')$ factors weakly as
a sequence of admissible stacky blow-ups.
Hence the expression $\{X, D\}$ only depends on $U$ by Proposition~\ref{pr-transform}.
Furthermore, if $V \subset U$ is a smooth closed substack,
then we can apply functorial embedded resolution of singularities on the closure of $V$ in $X$
compatible with~$D$.
Hence we may assume that we have a compactification $U \subset X$ such that the closure $Z$ of $V$
is smooth and has simple normal crossings with the exceptional locus $D = X \setminus U$.
In particular, the blow-up $\widetilde{X}$ of $X$ in $Z$ is a compactification of $U \setminus V$ with exceptional locus
given by the total transform~$\widetilde{D}$.
Hence Proposition~\ref{pr-transform} shows that~$\psi$ respects the scissors relations,
so the map is indeed well defined.

Now consider the sequence of maps
$$
\K_0(\DM^\sqp_k)
\xrightarrow{\psi} \K_0(\DM^\smpr_k) 
\xrightarrow{\varphi} \K_0(\DM_k). 
$$
It is clear that the composition $\varphi\circ\psi$ equals the map induced by the inclusion of categories,
which is an isomorphism by Proposition~\ref{pr-alt-def}.
Thus in order to show that $\varphi$ is an isomorphism,
it suffices to show that $\psi$ is surjective.
Assume that $X$ is an arbitrary irreducible, smooth, proper Deligne--Mumford stack.
Now choose a sequence of blow-ups
\begin{equation}
\label{eq-blow-up}
X_n \to \cdots \to X_0 = X
\end{equation}
in smooth centers such that $X_n$ is a smooth, projective Deligne--Mumford stack.
Such a sequence exists by \cite[Theorem~4.3]{choudhury2012}.
By the blow-up relations \eqref{eq-blow-up-rel},
we have
\begin{equation}
\label{eq-bl}
\{X_i\} = \{X_{i + 1}\} - \{E_{i+1}\} + \{Z_i\}
\end{equation}
where $Z_i \subset X_i$ is the center of the blow-up at step $i$ in the sequence \eqref{eq-blow-up}
and $E_{i+1} \subset X_{i+1}$ is the exceptional divisor.
By induction on the dimension of $X$,
we may assume that $\{E_{i+1}\}$ and $\{Z_i\}$ lie in the image of $\psi$.
Hence we may assume that the same holds for the right-hand side of \eqref{eq-bl}
by descending induction on~$i$.
In particular, the class $\{X_0\} = \{X\}$ lies in the image of $\psi$,
which concludes the proof.
\end{proof}

\bibliographystyle{myalpha}
\bibliography{main}

\end{document}